\documentclass[a4paper, runningheads]{llncs}

\usepackage[utf8]{inputenc}
\usepackage[T1]{fontenc}
\usepackage[pdftex,dvips]{graphicx}
\usepackage{ae,aecompl}

\usepackage{harvard}

\usepackage{graphicx}

\usepackage{ifthen}

\usepackage{amsmath}
\usepackage{amssymb}

\usepackage{mathpartir}

\usepackage{stmaryrd}

\usepackage[all,2cell,cmtip]{xy} \UseAllTwocells

\input{diagxy}

\newcommand{\op}[1]{\ensuremath{{#1}^\mathrm{op}}}

\newcommand{\thry}[1]{\ensuremath{\mathbb{#1}}}

\newcommand{\xalg}[1]{\ensuremath{\mathbf{#1}}}

\newcommand{\set}{\ensuremath{{Set}}}

\newcommand{\Rel}[1]{\ensuremath{\mathrm{Rel}(#1)}}

\newcommand{\instans}[2]{\ensuremath{#1\! :\! #2 \to<150>\mathbf{FinSet}}}

\newcommand{\kaninst}[1]{\ensuremath{1_#1}}
\newcommand{\kanskje}[2]{\ensuremath{\int_{#1}#2}}

\newcommand{\Trm}[2]{\ensuremath{\mathrm{Trm}_{#1}(#2)}}

\newcommand{\bb}[1]{\ensuremath{\left\llbracket#1 \right\rrbracket}}
\newcommand{\pair}[1]{\ensuremath{\left\langle {#1} \right\rangle}}

\newcommand{\cterm}[2]{\ensuremath{\left \{ {#1}\ \; \vrule \; \ {#2}\right \}}}
\newcommand{\kortinldpil}[3]{\ensuremath{#1\! :\! #2 \to/=>/<100>  #3 }}
\newcommand{\kortinlpil}[3]{\ensuremath{#1\! :\! #2 \to<150>  #3 }}

\newcommand\eval{\mathord{\uparrow}}

\newcommand{\showelim}{false}

\newcommand{\elim}[1]{ \ifthenelse{\equal{\showelim}{true} }{#1}{}}

\newcounter{myfigure}

\newcommand\Dthanks[1]{%
  \begingroup
  \renewcommand\thefootnote{}\footnote{#1}%
  \addtocounter{footnote}{-1}%
  \endgroup
}

\begin{document}
\author{Henrik Forssell, Håkon Robbestad Gylterud, David I. Spivak}

\institute{Stockholm University, Stockholm University, MIT}

\title{Type theoretical databases}

\maketitle
\Dthanks{\footnotesize Spivak acknowledges support from ONR grant N000141310260 and AFOSR grant FA9550-14-1-0031.}

 \begin{abstract}
  
We present a soundness theorem for a dependent type theory with context constants with respect to an indexed category of (finite, abstract) simplical complexes. The point of interest for computer science is that this category can be seen to represent tables in a natural way. Thus the category is a model for databases, a single mathematical structure in which all database schemas and instances (of a suitable, but sufficiently general form) are represented.  The type theory then allows for the specification of database schemas and instances, the manipulation of the same with the usual type-theoretic operations, and the posing of queries. 
  
 \end{abstract}

\section{Introduction}

Databases being, essentially,  collections of (possibly interrelated) tables of data, a foundational question is how to best represent such collections of tables mathematically in order to study their properties and  ways of manipulating them. The relational model, essentially treating tables as structures of first-order relational signatures, is a simple and powerful representation, the virtues of which we need not recapitulate. Nevertheless, areas exist in which the relational model is less adequate than in others. One familiar example is the question of how to represent partially filled out rows or missing information\footnote{see e.g.\ {\tt http://thethirdmanifesto.com/}}. Another, more fundamental perhaps, is how to relate instances of different schemas, as opposed to the relatively well understood relations between instances of the same schema. As comparison and mappings between data structured in different ways is an area of some importance for database theorists, for example in the settings of \cite{faginetal_dataexchange:2005}, \cite{lenzerini_dataiotegration:2002}, this suggests looking for alternative and supplemental ways of modeling tables more suitable to such ``dynamic'' settings. It seems natural, in that case, to try to model tables of different shapes as living in a single mathematical structure, facilitating  their manipulation across different schemas.

Formally, this paper presents a soundness theorem (Theorem \ref{theorem:soundness}) for a certain dependent type theory with respect to a rather simple category of (finite, abstract) simplicial complexes. The novelty is that the type theory has context constants, mirroring that our choice of ``display maps'' does not include all maps to the terminal object. From the database perspective, however, the interesting aspect is that this category can in a natural way be seen as a category of tables; collecting in a single mathematical structure---an \emph{indexed} or \emph{fibered} category---the totality of schemas and instances. 

This representation can be introduced as follows. Let a schema $S$ be presented as a finite set ${\bf A}$ of attributes and a set of relation variables over those attributes. One way of allowing for partially filled out rows is to assume that whenever the schema has a relation variable $R$, say over attributes $A_0,\ldots,A_n$, it also has relation variables over all non-empty subsets of $\{A_0,\ldots,A_n\}$. So a partially filled out row over $R$ is a full row over such a ``sub-relation'' of $R$. To this we add the requirement that the schema does not have two relation variables over exactly the same attributes\footnote{Coming from reasons having to do with simplicity and wanting to stay close to the view of tables as relations, the conceptual justification for this requirement is that if you have two tables with exactly the same column names, it should be natural to either collect them into one table or to rename some of the column names.}.  This requirement means that a relation variable can be identified with the set of its attributes, and together with the first requirement, this means that the schema can be seen as a downward closed sub-poset of the positive power set of the set of attributes ${\bf A}$. Thus a schema is an (abstract) \emph{simplicial complex}---a combinatorial and geometric object familiar from algebraic topology.

The key observation is now that an instance of the schema $S$ can also be regarded as a simplicial complex, by regarding the data as attributes and the tuples as relation variables. Accordingly, an instance over $S$ is a schema of its own, and the fact that it is an instance of $S$ is ``displayed'' by a certain projection to $S$. 
Thus the (structurally rich) category \thry{S} of finite simplicial complexes and morphisms between  them form a category of schemas which includes, at the same time, all instances of those schemas, the connection between schema and instance given by a collection of maps in \thry{S} called \emph{display} maps. 

As such, \thry{S} together with this collection $D$ of maps form a so-called \emph{display  map category} \cite{jacobs:99}, a notion originally developed in connection with categorical models of dependent type theory.   We show that $(\thry{S},D)$ is a model of a certain dependent type theory including the usual type-forming operations and with context constants. We indicate how any schema and instance (satisfying the above two requirements) can be specified in this type theory, and how the type-forming operations can then be used to build new instances and schemas (corresponding for instance to natural join and disjoint union), and thus to pose queries. Syntactically, that is in the type theory, schemas and instances are specified in the same way, in terms of  types and terms over a set of distinguished set of contexts corresponding to---and thus reflecting the special status of---single relation variable schemas (or \emph{relation schemas} $R[n]$ in the terminology of \cite{abiteboulhullvianu:95}).   

We focus, in the space available here, on the presentation of the model, the operations, the type theory, and the soundness theorem. Section \ref{section: the model} presents the model as an indexed category, defines the notion of a display morphism in this context, and shows how to pass from schema to instance and from instance to schema (and display morphism). Section \ref{Subsubsection: Examples} gives some  brief examples of schemas and instances as simplicial complexes (which we also use further down in Example \ref{Example: natural join} involving natural join, Section \ref{subsection:Instance specification as type introduction} concerning schema and instance specification, and Example \ref{example:queries} concerning queries).  Section \ref{Subsection: simplicial databases} present the semantic operations on schemas and instances interpreting the type theoretic operations, which in turn are presented together with the soundness theorem in Section \ref{Section: The type theory}. Future work will present expressivity and complexity analyses, the formulation of dependencies and constraints, as well as more formally presenting the relation to relational databases and to ``real'' instances (rather than the more ``structural'' instances we study here). Future work also includes exploiting the more geometric perspective on tables that this models offers (see \cite{spivak:simpdb}), and the modeling of schemas with multiple relation variables over the same attributes and instances with multiple keys representing the same data. 

Only knowledge of the very basic notions of category theory, such as category, functor, and natural transformation, is assumed.


\section{The model}\label{section: the model}

\subsection{Finite simplicial complexes and schemas}
\label{subsection: Finite simplicial complexes and schemas}

We consider both schemas and instances as (abstract, finite) simplicial complexes, with a certain family of maps ``displaying'' that a simplicial complex---the source of the map---is an instance of another complex---the target of the map---seen as a schema. We start with formal definitions of simplicial complexes, the corresponding \emph{simplicial schemas}, and the display maps. In what follows, all posets etc.\ are finite unless explicitly stated otherwise.

\begin{definition}\label{Def: based poset}
\begin{enumerate}
\item Let $X$ be a poset. A subset $B\subseteq X$ is called a {\em basis of $X$} if the following hold:
\begin{enumerate}
\item for all $x,y\in X$, one has $x\leq y \text{ if and only if } B_{\leq x}\subseteq B_{\leq y}$, where $B_{\leq x}=(\downarrow x)\cap B=\cterm{z\in B}{z\leq x}$;
\item no two elements of $B$ are comparable, i.e. for all $g,h\in B$ one has $g\not\leq h$; and
\item every element $x\in X$ has generators, i.e. $B_{\leq x}\neq\emptyset$.
\end{enumerate}
If $X$ has a basis, one sees easily that the basis is unique, and we say that $X$ is a {\em based poset}.

\item\label{Item: Based poset} Let $X$ be a based poset with basis $B$. Then define $X_n:=\cterm{x\in X}{|B_{\leq x}|=n+1}$. In particular, $X_0=B$.

\item\label{Item: simplicial set} A based poset $X$ is called a \emph{simplicial complex} if for all $x\in X$ and $Y\subseteq B_{\leq x}$ there exists $y\in X$ such that $B_{\leq y}=Y$. 

\item\label{Item: simplicial schema} A  poset $X$ is called a \emph{simplicial schema} if $\op{X}$---the poset obtained by reversing the ordering---is a simplicial complex. The elements of $X_0$ are called \emph{attributes} and the elements of $X_{n+1}$ are called relation variables, relation keys, or simply \emph{relations}.  

We consider a simplicial schema as a category and use arrows $\delta^x_y:x\to<150> y$ to indicate order. Thus the arrow $\delta^x_y$ exists iff $y\leq x$ in the simplicial complex \op{X}. We reserve the use of arrows to indicate order in the schema $X$ and $\leq$ to indicate the order in the  complex \op{X}. We use the notation $B_{\leq x}$ also in connection with schemas, where it means, accordingly, the set of attributes $A$ such that there is an arrow $\delta^x_A$. 

\item\label{Item: BPos morphism} Suppose that $X$ and $Y$ are based posets with bases $B$ and $C$ respectively. A (poset) morphism $f: X\to<150> Y$ is called {\em based} if for all $x\in X$, we have $f(B_{\leq x})=C_{\leq f(x)}$.  Let $\thry{S}$ be the category consisting of simplicial schemas and functors (poset morphisms) $f: X\to<150> Y$ such that $f: \op{X}\to<150> \op{Y}$ is a based poset morphism. Note that based poset morphisms are completely determined by their restriction to the basis. 

\item A morphism $f: X\to<150> Y$ of simplicial schemas is a \emph{display map} if $f$ restricts to a family of maps $f_n:X_n\rightarrow Y_n$ \textup{(}one could say that it `preserves dimension'\textup{)}. It is straightforward to see that this is equivalent to the condition that for all $x\in \op{X}$ the restriction $f\upharpoonright_{(\downarrow x)}: (\downarrow x)\rightarrow (\downarrow f(x))$ is an isomorphism of sets \textup{(}equivalently, of simplicial complexes\textup{)}. 

\end{enumerate}
\end{definition}

With respect to the usual notion of schema, a simplicial schema  $X$ can be thought of as given in the usual way by a finite set of attributes $X_0=\{A_0,\ldots,A_{n-1}\}$ and a set of relational variables $X=\{R_0,\ldots R_{m-1}\}$, each with a specification of column names in the form of a subset of $X_0$, but with the restrictions 1) that no two relation variables are over exactly the same attributes; and 2) for any (nonempty) subset of the attributes of a relation variable there exists a relation variable over (exactly) those attributes. In the sequel we shall drop the word ``simplicial'' and simply say ``complex'' and ``schema''. Any complex (and hence schema) can be seen as a subposet of a finite positive power set which is downward closed and contains all singletons via the injective function $X\to<150> P_+(X_0)$ defined by $x\mapsto B_{\leq x}$.  We freely use this perspective when convenient.

 The category \thry{S} contains in particular the $n$-simplices $\Delta_n$ and the face maps. Recall that the the $n$-simplex $\Delta_n$ is the complex given by the full positive power set on $[n]=\{0,\ldots,n\}$. A face map $d^n_i: \Delta_n\to<150>\Delta_{n+1}$ between two simplices is the based poset morphism defined by $k\mapsto k$, if $k<i$ and $k\mapsto k+1$ else. These satisfy the simplicial identity $d^{n+1}_i\circ  d^{n}_j = d^{n+1}_{j-1}\circ d^n_i$ if $i < j$.  As a schema,  $\Delta_n$ is the schema of a single relation on $n+1$ attributes named by numbers $0,\ldots,n$ (and all its ``generated'' sub-relations). A face map $d^n_i: \Delta_n\to<150>\Delta_{n+1}$ is the inclusion of the relation $[n+1] - \{i\}$  into $[n+1]$. These schemas and morphisms play a special role in Section \ref{Section: The type theory} where they are used to specify general schemas and instances.

\subsection{Relational instances}\label{subsection: relational instances}

Let $X$ be a (simplicial) schema, say with attributes $X_0=\{A_0,\ldots,A_{n-1}\}$.  A functor $F:X\to \xalg{FinSet}$ from $X$ to the category of finite sets and functions can be regarded as an instance of the schema $X$. 
For $x=\{A_{i_0},\dots,A_{i_{m-1}}\}\in X$, the set $F(x)$ can be regarded as a set of ``keys'' (or ``facts'' or ``row-names''). The ``value'' (or ``data''\footnote{Strictly speaking, ``data'' is somewhat misleading, as this notion of instance treats elements of, say, $F(A_0)$ and $F(A_1)$ as formally distinct. For example, the instances $F(A_0)=\{a\}$, $F(A_1)=\{b\}$  and $G(A_0)=\{a\}$, $G(A_1)=\{a\}$, of the schema $A_0,A_1$ with no relations, are isomorphic. An actual filling out of a table with data can be given as a function from, in this case,  the disjoint union $F(A_0)+F(A_1)$ to a domain of strings and numbers, say. We leave this extra level of structure to future work, and restrict attention to our more abstract notion of instance.}) $k[A_{i_j}]$ of such a key $k\in F(x)$ at attribute $A_{i_j}$ is then the element $k[A_{i_j}]:=F(d^x_{A_{i_j}})(k)$. Accordingly, $F(x)$ maps to the set of tuples $F(A_{i_0})\times\ldots\times F(A_{i_{m-1}})$ by $k\mapsto \pair{k[A_{i_0}],\ldots,k[A_{i_{m-1}}]}$.  For arbitrary $F$, this function is not 1--1, that is, there can be distinct keys with the same values at all attributes. We say that $F$ is a \emph{relational instance} if this does not happen:
\begin{definition}\label{Definition: Relational instances}
Let $X$ be a schema. A \emph{relational instance} is a functor  $F:X\to \xalg{FinSet}$ such that for all $x\in X$ the functions \cterm{\delta^{x}_{A}}{A\in B_{\leq x}} are jointly injective. Let \Rel{X} be the category of relational instances and natural transformations between them.
\end{definition}
We restrict interest in this paper to relational instances and refer to them in the sequel simply as \emph{instances}. It is clear that every relational instance is isomorphic to one where the keys are actually tuples, and we make use of this in Section \ref{Subsection: simplicial databases}.  
%
The relation between the condition that schemas have at most one relation variable over a set of attributes and the restriction to relational instances are displayed in the following: Let $X$ be a schema and $F:X\to<150>\xalg{FinSet}$ an arbitrary functor.  Recall, e.g.\ from \cite{maclane:98}, that the category of elements $\int_{{X}}F$ has objects \pair{x,a} with $x\in X$ and   $a\in F(x)$. A morphism $\delta^{\pair{x,a}}_{\pair{y,b}}: \pair{x,a}\to<150>\pair{y,b}$ is a morphism $\delta^{x}_{y}: x\to<150>y$ with $F(\delta^{x}_{y})(a)=b$. The projection $p:\kanskje{X}{F}\to<150> X$ is defined by $\pair{x,a} \mapsto x$ and $\delta^{\pair{x,a}}_{\pair{y,b}}\mapsto \delta^{x}_{y}$. We then have  

\begin{lemma}\label{Lemma: relational instances are based posets}
Let $X$ be a (simplicial) schema  and $F:X\to\xalg{FinSet}$ be a functor. Then $\int_X F$ is a (simplicial) schema  and $p:\int_X F\rightarrow X$ a display morphism if and only if $F$ is a relational instance.
\end{lemma}

\begin{definition}\label{Definition: Can inst, full tuple, induced schema}
\begin{enumerate}
\item Given a schema $X$  the functor $\kaninst{X}:X\to<150>\xalg{FinSet}$ defined by $x\mapsto \{x\}$ is the \emph{terminal instance} obtained by considering the attributes themselves to be the values and the relations themselves as the keys. 
\item  A \emph{full tuple} $t$ of an instance $I$ over schema $X$ is a natural transformation $t: \kaninst{X} \to/=>/<100> I $.    We write \Trm{X}{I} for the set of full tuples (indicating that we see them as terms type-theoretically). 
\item Given an instance \instans{I}{X} the \emph{induced schema } (over $X$) is the category of elements $\int_{{X}}I$.  The \emph{canonical projection} from the induced schema to $X$ is the projection $p:\kanskje{X}{I}\to<150> X$. 

\item Given a full tuple $t: \kaninst{X} \to/=>/<100> I $, the \emph{induced section} is the morphism $\hat{t}:X\to<150> \kanskje{X}{I}$ in \thry{S} defined by $x\mapsto \pair{x, t_x(x)}$. Notice that the induced section is always a display morphism. 
%
%
\end{enumerate}
\end{definition}

\subsubsection{Examples}\label{Subsubsection: Examples}

\begin{enumerate}
\item Let $S$ be the schema the attributes of which are $A,B,C$ and the relation variables $R:AB$ and $Q:BC$, with indicated column names. Let an instance $I$   be given by 

\begin{center}
\begin{tabular}{ l | c  r }
                     
  R & A & B \\
   \hline  
  1& a & b \\
  2 & a' & b \\
  
\end{tabular}
\quad
\begin{tabular}{ l | c  r }
                     
  Q & B & C \\
   \hline  
 1 & b & c \\
  2 & d & e \\
  
\end{tabular}
\end{center}

From a ``simplicial'' point of view, $S$ is the category

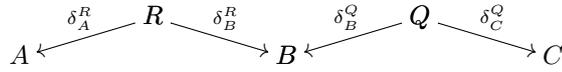
\[\bfig
\morphism<-500,-150>[R`A;\delta^R_A]
\morphism<500,-150>[R`B;\delta^R_B]
\morphism(1000,0)<-500,-150>[Q`B;\delta^Q_B]
\morphism(1000,0)<500,-150>[Q`C;\delta^Q_C]
\efig\] 

%

$I$ is the functor 

\[\bfig
\morphism<-800,-150>[I(R)= \{1,2\} `I(A)=\{a,a'\};I(\delta^R_A)]
\morphism<800,-150>[I(R)= \{1,2\}`I(B)=\{b,d\};I(\delta^R_B)]
\morphism(1600,0)<-800,-150>[I(Q)= \{1,2\}`I(B)=\{b,d\};I(\delta^Q_B)]
\morphism(1600,0)<800,-150>[I(Q)= \{1,2\}`I(C)=\{c,e\};I(\delta^Q_C)]
\efig\] 


with $I(\delta^R_A)(1)=a$, $I(\delta^R_B)(1)=b$ and so on. Notice that $I$ has precisely two full tuples; if we think of a full tuple as a selection of an element from each attribute such that the relevant projection of the resulting tuple has a matching key in all relations, then the two full tuples in $I$ are \pair{a,b,c} and \pair{a',b,c}. Finally, the induced schema \kanskje{S}{I} has attribute set and relation variables as follows, where we use subscript notation instead of pairing for readability: $\textnormal{Attributes}:\{a_A, a'_A, b_B, d_B, c_C, e_C\};\ \textnormal{Relations}: \{1_R\! :\! a_Ab_B, 2_R\! :\! a'_Ab_B,1_Q\! :\! b_Bc_C, 2_Q\! :\! d_Be_C\}$

\item Going from simplicial complexes to schemas, consider the 2-simplex $\Delta_2$ and an example functor   $J:\Delta_2\to<150>\xalg{FinSet}$ given by $J(\{0,1,2\})=\{\pair{a,b,c}\},\ J(\{0,1\})=\{\pair{a,b},\ \pair{a',b}\},\ J(\{1,2\})=\{\pair{b,c},\ \pair{d,e}\},\ J(\{0,2\})=\{\pair{a,c}, \pair{a',c}\},\ J(0)=\{a,a'\},\ J(1)=\{b,d\},\ J(2)=\{c,e\}$, where we use tuples for keys to avoid having to write out the $J(\delta)$s. Writing this up in table form we obtain:

 \begin{center}
\begin{tabular}{ c | c  c c }
                     
   & 0 & 1 & 2\\
   \hline  
  1& a & b &c \\

\end{tabular}
\quad
\begin{tabular}{ l | c  r }
                     
   & 0 & 1 \\
   \hline  
  1& a & b \\
  2 & a' & b \\
  
\end{tabular}
\quad
\begin{tabular}{ l | c  r }
                     
   & 0 & 2 \\
   \hline  
 1 & a & c \\
  2 & a' & c \\
  
\end{tabular}
\quad
\begin{tabular}{ l | c  r }
                     
   & 1 & 2 \\
   \hline  
 1 & b & c \\
  2 & d & e \\
  
\end{tabular}
\quad
\begin{tabular}{ l | c  }
                     
   & 0  \\
   \hline  
 1 & a  \\
  2 & a' \\
  
\end{tabular}
\quad
\begin{tabular}{ l | c   }
                     
   & 1 \\
   \hline  
 1 & b  \\
  2 & d  \\
  
\end{tabular}
\quad
\begin{tabular}{ l | c   }
                     
   &  2 \\
   \hline  
 1 &  c \\
  2 &  e \\
  
\end{tabular}
\end{center}

Notice that $J$ has precisely one full tuple. 

\item  There is a morphism of complexes $f:S\to<150>\Delta_2$ defined (and determined) by $A\mapsto 0$, $B\mapsto 1$, $C\mapsto 2$. The instance $I$ of $S$ is the restriction of $J$ to $S$ along $f$, that is, $I=J\circ f$. Such composition operations let us regard the category of schemas together with their instances as an indexed category. 

\end{enumerate}

\subsection{Simplicial databases}
\label{Subsection: simplicial databases}

We have, with the above, a (strict) functor $\Rel{-}:\op{\thry{S}}\to Cat$
defined by $X\mapsto \Rel{X}$ and $f:X\to<150>Y$ maps to $\Rel{f}=-\circ f:\Rel{Y}\to \Rel{X}$. We regard this strict, indexed category as a ``category of databases'' in which the totality of databases and schemas are collected. We give, first, the notation, definitions and equations we need for the operation  \Rel{f} for morphisms $f$ (which can syntactically be thought of as substitution), and follow with the instance-forming operations 0, 1, dependent sum and product, +, and identity. We devote some space to substitution and dependent sum, being perhaps the most interesting operations. For the remaining ones, only basic definitions (needed to interpret the type theory of Section \ref{Section: The type theory}) are given.

\subsubsection{Substitution}
\begin{definition}For $f:X\to<150>Y$ in \thry{S} and $J\in \Rel{Y}$ and $\kortinldpil{t}{1_Y}{J}$ in $\Trm{Y}{J}$:
\begin{enumerate}
\item We write $J[f]:=\Rel{f}(J)=J\circ f$. With \kortinlpil{g}{Z}{X} we have $J[f][g]=J[f\circ g] = J\circ f\circ g$. 
\item Define $t[f]\in \Trm{X}{J[f]}$  by $x\mapsto t(f(x))\in J[f](x)$. Note that  $t[f][g]=t[f\circ g]$.
\item With $\kortinlpil{p_J}{\kanskje{Y}{J}}{X}$ the canonical projection,  let $\kortinldpil{v_J}{1_{\kanskje{Y}{J}}}{ J[p_J]}$ be the full tuple defined by $\pair{y,a}\mapsto a$. (We elsewhere leave subscripts on $v$ and $p$ determined by context.)
\item Denote by \kortinlpil{\tilde{f}}{\kanskje{X}{J[f]}}{\kanskje{Y}{J}} the schema morphism defined by $\pair{x,a}\mapsto \pair{f(x),a}$. Notice that if $f$ is display, then so is $\tilde{f}$.
\end{enumerate}
\end{definition}

\begin{lemma}\label{Lemma: Soundness for substitution}
The following equations hold:
\begin{enumerate}
%
\item For $X$ in \thry{S} and $I\in \Rel{X}$ and $t\in \Trm{X}{I}$ we have $p\circ \hat{t}=\mathrm{id}_X$ and $t=v[\hat{t}]$.
\item For $f:X\to<150>Y$ in \thry{S} and $J\in \Rel{Y}$ $t\in \Trm{Y}{J}$ we have 
\begin{enumerate}
\item $\kortinlpil{p\circ \tilde{f}= f\circ p}{\kanskje{X}{J[f]}}{Y}$;
\item  $\kortinlpil{\tilde{f}\circ\widehat{t[f]}=\hat{t}\circ f}{X}{\kanskje{Y}{J}}$; and
\item   $\kortinldpil{v_J[\tilde{f}]=v_{J[f]}}{1_{\kanskje{X}{J[f]}}}{J[f][p]}$.
\end{enumerate}

\item For $f:X\to<150>Y$ and $g:Y\to<150>Z$  in \thry{S} and $J\in \Rel{Z}$ we have  $\widetilde{g\circ f}=\tilde{g}\circ \tilde{f}$.
\item For $X \in \thry{S}$ and $I\in \Rel{X}$ we have  $\tilde{p}\circ \hat{v}=\mathrm{Id}_{\kanskje{X}{I}}$. 
\end{enumerate}
\end{lemma}

\subsubsection{On tuplification and display maps.} We have so far considered arbitrary morphisms \kortinlpil{f}{X}{Y} in \thry{S}. In what follows we shall restrict to display ones, noting that all context morphisms of the type theory in Section \ref{Section: The type theory} are interpreted as display maps. Furthermore, we shall assume, at least implicitly, that all instances are on \emph{tuple form}, in the sense that for \kortinlpil{J}{X}{\xalg{FinSet}} the elements in $J(x)$ are in fact elements in the product $\Pi_{A\in B_{\leq x}}J(A)$. We identify a singleton tuple with its element and write $\star$ for the empty tuple. It is clear that any instance is canonically isomorphic to one on tuple form; the \emph{tuplification} of the instance, so to say. Moreover, if two instances are isomorphic, on tuple form, and agree on attributes, then they must be identical. It is this property that we need for certain of the equations below. That being said, it is often more convenient to work with instances not on this form, and we frequently do so, pointing to the fact that the ``tuplification'' is just a canonical rewriting. Finally, we note the following connection between tuple form and display maps: 

\begin{lemma}\label{Lemma: display is tuple preserving}
Let  $f:X\to<150>Y$ in \thry{S}. Then $f$ is display if and only if for all $J\in \Rel{Y}$ on tuple form $J[f]$ is also on tuple form.
%
\end{lemma}   
%

\subsubsection{Dependent product:}

Let $X\in\thry{S}$, $J\in \Rel{X}$, and $G\in\Rel{\kanskje{X}{J}}$.  We define the instance $\kortinlpil{\Pi_JG}{X}{\xalg{FinSet}}$  as the right Kan-extension of $G$ along $p$. Explicitly, we preliminary construct the following instance, not on tuple form. Let $\kortinlpil{\Xi_JG}{X}{\xalg{FinSet}}$ be defined as follows. For $x\in X$ define $P_x=\cterm{\pair{y,a}}{y\leq x,\ a\in J(y)}$ and set $\Xi_JG(x)$ to be the set \cterm{(c_{y,a})_{\pair{y,a}\in P_x}}{c_{y,a}\in G(y,a)} of families satisfying the condition that

\begin{equation} \textnormal{for all}\ \delta^{{y,a}}_{z,b}\ \textnormal{in}\ \kanskje{X}{J}\ \textnormal{with}\ y\leq x\ \textnormal{we have}\ G(\delta^{{y,a}}_{z,b})(c_{y,a})=c_{z,b}. 
\end{equation}
Thus an element $c \in \Xi_JG(x)$ is a function $c$ assigning for each $y\leq x$ and $a\in J(y)$ an element $c_{y,a}$ in $G(y,a)$ such that (1) is fulfilled.
For $x'\leq x$, we have $P_{x'}\subseteq P_x$ and the function $\kortinlpil{\Xi_JG(\delta^x_{x'})}{\Xi_JG(x)}{\Xi_JG(x')}$ sends a family $(c_{y,a})_{P_x}$ to its projection on  $P_{x'}$. In particular, for $x_0\in G_{\leq x}$ we have that $\Xi_JG(\delta^x_{x_0})((c_{y,a})_{P_x})$ is the projection on $y=x_0$. Since $G$ is relational, $c_{y,a}\in G(y,a)$ is determined by the family $(G(\delta^{y,a}_{y_0,\delta^{y}_{y_0}(a)})(c_{y,a}))_{y_0\in B_{\leq y}}$ (recall the notation $B_{\leq y}$ for the set of attributes of $y$). Thus a family in $\Xi_JG(x)$ is determined by its values on       $\Xi_JG(\delta^x_{x_0})$ as $x_0$ runs through $B_{\leq x}$, and we have;

\begin{lemma}
$\kortinlpil{\Xi_JG}{X}{\xalg{FinSet}}$ is relational.
\end{lemma}

Let $X\in\thry{S}$, $J\in \Rel{X}$,  $G\in\Rel{\kanskje{X}{J}}$, and $\kortinldpil{t}{1_{\kanskje{X}{J}}}{G}$ a full tuple. Then for every $x\in X$ we have that the family $(t(y,a))_{P_x}$ is an element in $\Xi_JG(x)$, since the condition (1) is satisfied. Furthermore, for $\delta^x_{x'}$ we have that $\Xi_JG(\delta^x_{x'})((t(y,a))_{P_x})=(t(y,a))_{P_x'}$. Thus $t$ determines a full tuple which we call $\lambda t\in\Trm{X}{\Xi_JG}$.

Next, given $X\in\thry{S}$, $J\in \Rel{X}$,  $G\in\Rel{\kanskje{X}{J}}$, and $s\in \Trm{X}{\Xi_JG}$, define a full tuple $\mathrm{Ap}_s\in \Trm{\kanskje{X}{J}}{G}$ by $\pair{x,a}\mapsto s(x)_{x,a}$ (the \pair{x,a}'th element of the family $s(x)$---it is straightforward to verify that this does define a full tuple by using condition (1) and the fact that $s$ is a full tuple).

Now, this preliminary definition is not of tuple form, so the final definition of  $\Pi_JG$ is the  tuplification of $\Xi_JG$. The definitions of $\lambda$ and $\mathrm{Ap}$ are changed accordingly, i.e.\ along the isomorphism. Note that for an attribute $A\in X_0$ we have that $\Pi_JG(A)$ is the product $\Pi_{a\in J(A)}G(A,a)$.  We have then the following equations.

\begin{lemma}\label{Lemma: Pi regler}
 Let $\kortinlpil{f}{Z}{X}$  be a display morphism in  \thry{S}, $J\in \Rel{X}$,  $G\in\Rel{\kanskje{X}{J}}$, and $t\in \Trm{(\kanskje{X}{J})}{G}$.
 \begin{enumerate}
 \item $\mathrm{Ap}_{\lambda t}= t$.
 \item $(\Pi_JG)[f]=\Pi_{J[f]}G[\tilde{f}]$.
 \item $(\lambda t)[f]=\lambda(t[\tilde{f}])$.
 \end{enumerate}
%
%
%
%
\end{lemma}

\begin{example}\label{Example: natural join}
Consider the schema $S$ of Section \ref{Subsubsection: Examples}, which we can give as an instance of $\Delta_2$ as (ignoring tuplification for readability) $\kortinlpil{S}{\Delta_2}{\xalg{FinSet}}$ by $S(0)=\{A\}$, $S(1)=\{B\}$, $S(2)=\{C\}$, $S(01)=\{R\}$, $S(12)=\{Q\}$, and $S(02)=S(012)=\emptyset$. Notice that, modulo the isomorphism between $S$ as presented in  \ref{Subsubsection: Examples} and  \kanskje{\Delta_2}{S},  the morphism $\kortinlpil{f}{S}{\Delta_2}$ of  \ref{Subsubsection: Examples}  is the canonical projection \kortinlpil{p}{\kanskje{\Delta_2}{S}}{\Delta_2}. Next we have $I\in \kanskje{\Delta_2}{S}$ as  (in tabular form,  using subscript instead of pairing for elements in $ \kanskje{\Delta_2}{S}$, and omitting the three single-column tables)

\begin{center}
\begin{tabular}{ l | c  c}
                     
  $R_{01}$ & $A_0$ & $B_1$ \\
   \hline  
  & a & b \\
   & a' & b \\
  
\end{tabular}
\quad
\begin{tabular}{ l | c  c }
                     
  $Q_{12}$ & $B_1$ & $C_2$ \\
   \hline  
  & b & c \\
   & d & e \\
  
\end{tabular}
\end{center}
Then by unpacking the definition, one sees that $\Pi_{S}I$ is, in tabular form,

 \begin{center}
\begin{tabular}{ c | c  c c }
                     
   & 0 & 1 & 2\\
   \hline  
  & a & b &c \\
  & a' & b & c \\ 
  
\end{tabular}
\quad
\begin{tabular}{ l | c  c }
                     
   & 0 & 1 \\
   \hline  
  & a & b \\
   & a' & b \\
  
\end{tabular}
\quad
\begin{tabular}{ l | c  c }
                     
   & 0 & 2 \\
   \hline  
  & a & c \\
   & a' & c \\
  & a & e \\
   & a' & e \\  
\end{tabular}
\quad
\begin{tabular}{ l | c  c }
                     
   & 1 & 2 \\
   \hline  
  & b & c \\
   & d & e \\
  
\end{tabular}
\quad
\begin{tabular}{ l | c   }
                     
   & 0  \\
   \hline  
  & a  \\
  & a'  \\
  
\end{tabular}
\quad
\begin{tabular}{ l | c   }
  &  1 \\
   \hline  
  &  b \\
   &  d \\
  
\end{tabular}
\quad
\begin{tabular}{ l | c  }
                     
   & 2 \\
   \hline  
  &  c \\
   &  e \\
  
\end{tabular}

\end{center}

\noindent This exemplifies the link between the dependent product operation and natural join.

\end{example}


\begin{description}
\item[0 and 1 instances:]
Given $X\in \thry{S}$ the terminal instance $1_X$ has already been defined.  The \emph{initial} instance $0_X$ is the constant $0$ functor, $x\mapsto \emptyset$. Note that \kanskje{X}{0_X} is the empty schema.
\item[Dependent sum:]
\label{sigma-modell}
Let $X\in\thry{S}$, $J\in \Rel{X}$, and $G\in\Rel{\kanskje{X}{J}}$.  We define the instance $\kortinlpil{\Sigma_JG}{X}{\xalg{FinSet}}$  (up to tuplification) by $x\mapsto\cterm{\pair{a,b}}{a\in J(x),\ b\in G(x,a)}$. For $\delta^x_y$ in $X$, $\Sigma_JG(\delta^x_y)(a,b)=\pair{\delta^x_y(a),\delta^{x,a}_{y,\delta^x_y(a)}(b)}$. 
%
We leave to the reader the straightforward verification that $\Sigma_JG$ is relational and that the stability-under-substitution equation $(\Sigma_JG)[f]=\Sigma_{J[f]}G[\tilde{f}]$
holds.
%
%
%
\item[Identity:]
\label{id-modell}
Given $X\in \thry{S}$ and $J\in\Rel{X}$ the \emph{Identity instance} $\mathrm{Id}_J\in \Rel{\kanskje{\kanskje{X}{J}}{J[p]}}$ is defined, up to tuplification, by $\pair{\pair{x,a},b}\mapsto \star$ if $a=b$ and $\pair{\pair{x,a},b}\mapsto \emptyset$ else.  The full tuple $\operatorname{refl}\in \Trm{(\kanskje{X}{J})}{\mathrm{Id}_J[\hat{v}]}$ is defined by $\pair{x,a}\mapsto \star$. 
\item[Disjoint union:]
\label{sum-modell}
Given $X\in \thry{S}$ and $I,J\in\Rel{X}$, the instance  $I+J\in\Rel{X}$ is defined by $x\mapsto \cterm{\pair{n,a}}{\textnormal{Either}\ n=0\wedge a\in I(x)\ \textnormal{or}\ n=1\wedge a\in J(x)}$. We have full tuples $\operatorname{left}\in \Trm{\kanskje{X}{I}}{(I+J)[p]}$ defined by $\pair{x,a}\mapsto \pair{0,a}$ and $\operatorname{right}\in \Trm{\kanskje{X}{J}}{(I+J)[p]}$ defined by $\pair{x,a}\mapsto \pair{1,a}$.
\end{description}

\section{The type theory}\label{Section: The type theory}

We introduce a Martin-Löf style type theory \cite{mltt}, with explicit substitutions, extended with context and substitution constants representing simplices and face maps. The substitution and context part of the type theory is essentially that of Categories with families \cite{hofmann97}, except for a novel set of context and substitution constants. In the following, we display the basic rules for substitution, context extension and for the type contructions $\Pi$, $\Sigma$, and $+$. Listed here are all rules introducing new terms, and selected, important equations. We also omit some elimination rules, but refer to the literature on type theory, such as \cite{hofmann97,nordstrom-petersson}.



For each collection of rules we give the intended interpretation in the model. The interpretation is given by the operation $\bb{-}$, defined recursively throught this section. This is then summed up in the soundness theorem in the end.

\subsection{Judgements}

The type system has the following eight judgements, with intended interpretations.

\begin{center}
\begin{tabular}{l|l}
 Judgement & Interpretation \\ \hline
$?\ \mathtt{context}$ & $?$ is a schema\\
$\Gamma\vdash {?}\ \mathtt{type}$ & $?$ is an instance of the schema $\Gamma$\\
$\Gamma\vdash {?} : A$ & ? is an full tuple in the instance $A$\\
$? : \Gamma \to \Lambda$ & ? is a (display) schema morphism\\
$\Gamma \equiv \Lambda$ & $\Gamma$ and $\Lambda$ are equal schemas\\
$\Gamma\vdash  A \equiv B$ & $A$ and $B$ are equal instances of $\Gamma$\\
$\Gamma\vdash t \equiv u: A$ & $t$ and $u$ are equal full tuples in $A$\\
$\sigma \equiv \tau : \Gamma \to \Lambda$ & the morphisms $\sigma$ and $\tau$ are equal
\end{tabular}
\end{center}

\subsection{Substitutions}

The following axioms state that contexts and substitutions form a category, acting on the types and elements. Below the axioms are stated the intended interpretations cf.\ Section \ref{section: the model}


\begin{mathpar} 
\inferrule*[right=subst-comp]
  {\Gamma,\Delta,\Theta\ \mathtt{context} \\ \sigma : \Gamma \to \Delta \\ \tau : \Delta \to \Theta}
  {\tau \circ \sigma : \Gamma \to \Theta}

\elim{
\inferrule*[right=subst-assoc]
  {\Gamma,\Delta, \Theta, \Lambda\ \mathtt{context} \\ \sigma : \Gamma \to \Delta \\ \tau : \Delta \to \Theta \\ \upsilon : \Theta \to \Lambda}
  {(\upsilon \circ \tau) \circ \sigma \equiv \upsilon \circ (\tau \circ \sigma): \Gamma \to \Lambda}
}


\inferrule*[right=subst-id]
    { \Gamma\ \mathtt{context}}
{id_\Gamma : \Gamma \to \Gamma }


\elim{
\inferrule*[right=subst-id-left]
    { \Gamma,\Delta\ \mathtt{context} \\ \sigma : \Gamma \to \Delta}
{id_\Gamma \circ \sigma \equiv \sigma: \Gamma \to \Delta }
}


\elim{
\inferrule*[right=subst-id-right]
    { \Gamma,\Delta\ \mathtt{context} \\ \sigma : \Gamma \to \Delta}
{\sigma \circ id_\Gamma  \equiv \sigma: \Gamma \to \Delta }
}


\inferrule*[right=type-transfer]
    { \elim{\Gamma, \Delta\ \mathtt{context} \\} \Gamma \vdash A \ \mathtt{type} \\ \sigma : \Delta \to \Gamma }
{\Delta \vdash A[\sigma]\ \mathtt{type} }


\elim{
\inferrule*[right=type-transfer-comp]
    { \elim{\Gamma, \Delta, \Theta\ \mathtt{context} \\} \Gamma \vdash A \ \mathtt{type} \\ \sigma : \Delta \to \Gamma \\ \tau : \Theta \to \Delta }
{\Theta \vdash A[\sigma \circ \tau] \equiv A[\sigma][\tau] }
}


\elim{
\inferrule*[right=type-transfer-id]
    { \elim{\Gamma\ \mathtt{context} \\} \Gamma \vdash A \ \mathtt{type} }
{\Gamma \vdash A[id_\Gamma] \equiv A }
}


\inferrule*[right=term-transfer]
    {\elim{ \Gamma, \Delta\ \mathtt{context} \\ \Gamma \vdash A \ \mathtt{type} \\} \Gamma \vdash a : A \\ \sigma : \Delta \to \Gamma}
{\Delta \vdash a[\sigma] : A[\sigma] }


\elim{
\inferrule*[right=term-transfer-comp]
    { \elim{\Gamma, \Delta, \Theta\ \mathtt{context} \\ \Gamma \vdash A \ \mathtt{type} \\} \Gamma \vdash a : A \\  \sigma : \Delta \to \Gamma \\ \tau : \Theta \to \Delta }
{\Theta \vdash a[\sigma \circ \tau] \equiv a[\sigma][\tau] : A[\sigma][\tau] }
}

\elim{
\inferrule*[right=term-transfer-id]
    { \elim{\Gamma\ \mathtt{context} \\ \Gamma \vdash A \ \mathtt{type} \\} \Gamma \vdash a : A }
{\Delta \vdash a[id_\Gamma] \equiv a : A }
}

\end{mathpar}


\begin{align*}
\bb{\sigma \circ \tau} &= \bb{\sigma} \circ \bb{\tau} &
\bb{A[\sigma]} &= \bb{A} \circ \bb{\sigma} &
\bb{Id_\Gamma} &= Id_{\bb{\Gamma}}  &
\bb{a[\sigma]} &= \bb{a} \bb{\sigma}
\end{align*}

\subsection{Context extension}

The following rules detail how types in a context extends it.


\begin{mathpar}
\inferrule*[right=cont-ext]
    { \Gamma\ \mathtt{context} \\ \Gamma \vdash A \ \mathtt{type}}
{\Gamma. A \ \mathtt{context}}


\inferrule*[right=projection]
    { \elim{ \Gamma\ \mathtt{context} \\} \Gamma \vdash A \ \mathtt{type}}
{{\downarrow} : \Gamma. A \to \Gamma}
\end{mathpar}
\begin{mathpar}

\inferrule*[right=variable]
    { \elim{ \Gamma\ \mathtt{context} \\} \Gamma \vdash A \ \mathtt{type}}
{\Gamma. A \vdash v : A[\downarrow]}


\inferrule*[right=evaluation]
    { \elim{\Gamma\ \mathtt{context} \\ \Gamma \vdash A \ \mathtt{type} \\} \Gamma \vdash a : A}
{{a\eval} : \Gamma \to \Gamma. A}


\inferrule*[right=section]
    { \elim{\Gamma\ \mathtt{context} \\ \Gamma \vdash A \ \mathtt{type} \\ }\Gamma \vdash a : A}
{\downarrow \circ (a\eval) \equiv id_\Gamma: \Gamma \to \Gamma.}


\inferrule*[right=var-eval]
    { \elim{\Gamma\ \mathtt{context} \\ \Gamma \vdash A \ \mathtt{type} \\} \Gamma \vdash a : A}
{\Gamma \vdash v[a\eval] \equiv a : A}


\inferrule*[right=lifting]
    { \elim {\Gamma,\Delta\ \mathtt{context} \\ }\Gamma \vdash A \ \mathtt{type} \\ \sigma : \Delta \to \Gamma }
{{\sigma.A} : \Delta.A[\sigma] \to \Gamma. A}


\inferrule*[right=project-lift]
    { \elim{\Gamma,\Delta\ \mathtt{context} \\} \Gamma \vdash A \ \mathtt{type} \\ \sigma : \Delta \to \Gamma }
{{\downarrow} \circ (\sigma.A) \equiv \sigma \circ {\downarrow} : \Delta.A[\sigma] \to \Gamma}


\inferrule*[right=variable-lift]
    { \elim{\Gamma,\Delta\ \mathtt{context} \\ }\Gamma \vdash A \ \mathtt{type} \\ \sigma : \Delta \to \Gamma }
{ v[f.A] \equiv v : \Delta.A[\sigma]}


\inferrule*[right=eval-lift]
    { \elim{\Gamma,\Delta\ \mathtt{context} \\ \Gamma \vdash A \ \mathtt{type} \\} \Gamma \vdash a : A \\ \sigma : \Delta \to \Gamma }
{ (\sigma.A) \circ (a[\sigma]\eval) \equiv a\eval \circ \sigma : \Delta \to \Gamma.A}


\inferrule*[right= compose-lift]
    { \elim{\Gamma,\Delta,\Theta\ \mathtt{context} \\} \Gamma \vdash A \ \mathtt{type} \\ \sigma : \Delta \to \Gamma \\ \tau : \Theta \to \Delta}
{(\sigma.A) \circ (\tau.A[\sigma]) \equiv (\sigma \circ \tau) . A : \Theta.A[\sigma \circ \tau] \to \Gamma.A}


\inferrule*[right=diagonal-lift]
    { \elim{\Gamma,\Delta\ \mathtt{context} \\} \Gamma \vdash A \ \mathtt{type} }
{ \downarrow.A \circ v\eval \equiv Id_{\Gamma.A}: \Gamma.A \to \Gamma.A}

\end{mathpar}


\begin{align*}
\bb{\Gamma.A} &= \int_{{\bb{\Gamma}}}\bb{A} &
\bb{\downarrow} &= p &
 \bb{v} &= v &
\bb{a\eval} &= \widehat{\bb a} &
\bb{\sigma.A} &= \widetilde{\bb \sigma}
\end{align*}

\elim{
\subsubsection{Reading explicit substitutitions} 

The visually inclined reader will find it helpful to draw diagrams, to understand the explicit substitutions. The more operationally inclined reader might find the following heuristics and examples useful.

The basic substitutions, $\downarrow$ and $t \eval$ and $\sigma.A$, can be seen as modifying the current context. The projection $\downarrow$ removes an assumption. The evaluation $t \eval$, introduces an assumption. The lifting $\sigma.A$ first removes an assumtion of $A[\sigma]$, then performs the action specified by $\sigma$ and then introduces the assumption $A$. The reading order of substitutions is right to left.

\begin{example}
Assume $\Gamma\ \mathtt{context}$, and that $\Gamma \vdash A\ \mathtt{type}$, and $\Gamma \vdash C\ \mathtt{type}$. Then cosider the judgement $\Gamma.C \vdash A[\downarrow]\ \mathtt{type}$, which is a derivable instance of the type transfer rule. Following the above heuristics, the $[\downarrow]$, removes the $.C$ from the context, allowing us to access $A$ from $\Gamma$.

Let us now extend the example with $\Gamma.A \vdash B\ \mathtt{type}$, and $\Gamma.C \vdash D\ \mathtt{type}$, and $\Gamma.C \vdash a:A[\downarrow]$. Consider now the judgement, $\Gamma.C.D \vdash B[\downarrow.A][a\eval][\downarrow]$. Reading from right to left, we obtain the heuristic: $\downarrow$ removes the assumption $D$(leaving $\Gamma.C$). In this context, $a$ is of type $A[\downarrow]$, so $a\eval$ introduces the assumption $A[\downarrow]$. Finally $({\downarrow}.A)$ removes the assumption $A[\downarrow]$, then applies $\downarrow$ which removes the assumption $C$, and adds the assumption $A$, leaving the context $\Gamma.A$, where indeed $B$ is a type.

\end{example}
}

\subsection{The $\Pi$-type}

\begin{mathpar}
\inferrule*[right=$\Pi$-form]
 { \Gamma \ \mathtt{context} \\ \Gamma \vdash A\ \mathtt{type} \\ \Gamma.A\vdash B\ \mathtt{type}}
{\Gamma \vdash \Pi_A B\ \mathtt{type} }

\inferrule*[right=$\Pi$-subst]
 { \Gamma,\Delta \ \mathtt{context} \\ \Gamma \vdash A\ \mathtt{type} \\ \Gamma.A\vdash B\ \mathtt{type} \\ \sigma : \Delta \to \Gamma}
{\Gamma \vdash \Pi_A B [\sigma] \equiv \Pi_{A[\Sigma]}B[\sigma.A] }

\inferrule*[right=$\Pi$-intro]
 { \Gamma.A \vdash b : B}
{\Gamma \vdash \lambda b : \Pi_A B }

\inferrule*[right=$\Pi$-elim]
 { \Gamma \vdash f : \Pi_A B}
{\Gamma .A \vdash \operatorname{apply} f : B }

\inferrule*[right=$\Pi$-comp]
 { \Gamma.A \vdash b : B}
{\Gamma .A \vdash \operatorname{apply}(\lambda b) \equiv b : B }

\end{mathpar}

\begin{align*}
  \bb{\Pi_A B} &= \Pi_{\bb{A}}\bb{B} &
  \bb{\lambda f} &= \lambda \bb{f} &
  \bb{\operatorname{apply}} &= Ap
\end{align*}

\subsection{Other types}

These types correspond to the constructions with the same names in subsection \ref{sigma-modell}.

\begin{description}
\item[The $\Sigma$-type]

\begin{mathpar}
\inferrule*[right=$\Sigma$-form]
 { \Gamma \ \mathtt{context} \\ \Gamma \vdash A\ \mathtt{type} \\ \Gamma.A\vdash B\ \mathtt{type}}
{\Gamma \vdash \Sigma_A B\ \mathtt{type} }

\inferrule*[right=$\Sigma$-intro]
{ \Gamma.A\vdash B\ \mathtt{type}}
{\Gamma.A.B \vdash \operatorname{pair} : \Sigma_A B [\downarrow][\downarrow]}

\elim{
\inferrule*[right=$\Sigma$-elim]
 { \Gamma.\Sigma_A B \vdash C\ \mathtt{type} \\ \Gamma.A.B \vdash c_0 : C[(\downarrow \circ \downarrow).\Sigma_A B][\operatorname{pair}\eval]}
{\Gamma. \Sigma_A B \vdash \operatorname{rec}_\Sigma \, c_0 : C}

\inferrule*[right=$\Sigma$-comp]
 { \Gamma.\Sigma_A B \vdash C\ \mathtt{type} \\ \Gamma.A.B \vdash c_0 : C[(\downarrow \circ \downarrow).\Sigma_A B][\operatorname{pair}\eval]}
{\Gamma.A.B \vdash (\operatorname{rec}_\Sigma \, c_0)[(\downarrow \circ \downarrow).\Sigma_A B][\operatorname{pair}\eval]  \equiv  c : C[(\downarrow \circ \downarrow).\Sigma_A B][\operatorname{pair}\eval ] }
}
\end{mathpar}

\elim{
\begin{align*}
  \bb{\Sigma_A B} &= \Sigma_{\bb{A}}\bb{B} &
  \bb{\operatorname{pair}} &= \operatorname{pair} \elim{&
  \bb{\operatorname{rec}_\Sigma} &= \operatorname{rec}_\Sigma}
\end{align*}}

\item[The indentity type]

\begin{mathpar}
\inferrule*[right=$Id$-form]
 { \Gamma \ \mathtt{context} \\ \Gamma \vdash A\ \mathtt{type}}
{\Gamma .A .A[\downarrow]\vdash Id_A\ \mathtt{type} }

\inferrule*[right=$Id$-intro]
{ \Gamma \vdash A\ \mathtt{type}}
{\Gamma.A \vdash \operatorname{refl} : Id_A [v \eval]}

\elim{
\inferrule*[right=$Id$-elim]
 { \Gamma.A.A[\downarrow].Id_A \vdash C\ \mathtt{type} \\ \Gamma.A \vdash c_0 : C[(v \eval).Id_A][\operatorname{refl}\eval]}
{\Gamma.A.A[\downarrow].Id_A \vdash \operatorname{rec}_{Id} \, c_0 : C}

\inferrule*[right=$Id$-comp]
 { \Gamma.A.A[\downarrow].Id_A \vdash C\ \mathtt{type} \\ \Gamma.A \vdash c_0 : C[(v \eval).Id_A][\operatorname{refl}\eval]}
{ \Gamma.A \vdash (\operatorname{rec}_{Id} \, c_0)[(v \eval).Id_A][\operatorname{refl}\eval] \equiv c_0 : C[(v \eval).Id_A][\operatorname{refl}\eval] }
}

\end{mathpar}

\elim{
\begin{align*}
  \bb{Id_A} &= Id_{\bb{A}}\bb{B} &
  \bb{\operatorname{refl}} &= \operatorname{refl} \elim{&
  \bb{\operatorname{rec}_{Id}} &= \operatorname{rec}_{Id}}
\end{align*}
}

\item[The $+$-type]

\begin{mathpar}
\inferrule*[right=$+$-form]
 {\elim{ \Gamma \ \mathtt{context} \\} \Gamma \vdash A\ \mathtt{type} \\ \Gamma \vdash B\ \mathtt{type}}
{\Gamma \vdash A+B\ \mathtt{type} }

\inferrule*[right=$+$-intro-left]
{ \elim{\Gamma \ \mathtt{context} \\ \Gamma \vdash A\ \mathtt{type} \\ \Gamma \vdash B\ \mathtt{type}}}
{\Gamma.A \vdash \operatorname{left} : (A + B)[\downarrow] }

\inferrule*[right=$+$-intro-right]
{ \elim{\Gamma \ \mathtt{context} \\ \Gamma \vdash A\ \mathtt{type} \\ \Gamma \vdash A\ \mathtt{type}}}
{\Gamma.B \vdash \operatorname{right} : (A + B)[\downarrow]}
\elim{
\inferrule*[right=$+$-elim]
 { \Gamma.(A+B) \vdash C\ \mathtt{type} \\ \Gamma.A \vdash c_0 : C[(\downarrow).(A+B)][\operatorname{left}] \\ \Gamma.B \vdash c_1 : C[(\downarrow).(A+B)][\operatorname{right}] }
{ \Gamma.(A+B) \vdash \operatorname{rec}_{+}\, c_0 \, c_1 : C}

\inferrule*[right=$+$-comp-left]
 {  \Gamma.(A+B) \vdash C\ \mathtt{type} \\ \Gamma.A \vdash c_0 : C[(\downarrow).(A+B)][\operatorname{left} \eval] \\ \Gamma.B \vdash c_1 : C[(\downarrow).(A+B)][\operatorname{right} \eval] }
{ \Gamma.A \vdash (\operatorname{rec}_{+}\, c_0 \, c_1)[(\downarrow).(A+B)][\operatorname{left}\eval] \equiv c_0 : C[(\downarrow).(A+B)][\operatorname{left}\eval] }

\inferrule*[right=$+$-comp-right]
 {  \Gamma.(A+B) \vdash C\ \mathtt{type} \\ \Gamma.A \vdash c_0 : C[(\downarrow).(A+B)][\operatorname{left} \eval] \\ \Gamma.B \vdash c_1 : C[(\downarrow).(A+B)][\operatorname{right} \eval] }
{ \Gamma.B \vdash (\operatorname{rec}_{+}\, c_0 \, c_1)[(\downarrow).(A+B)][\operatorname{right}\eval] \equiv c_1 : C[(\downarrow).(A+B)][\operatorname{right}\eval] }
}
\end{mathpar}

\item[The $0$-type]

\begin{mathpar}
\inferrule*[right=$0$-formation]
{\Gamma\ \mathtt{context}}
{\Gamma \vdash 0\ \mathtt{type}}

\inferrule*[right=$0$-elimination]
{\Gamma.0 \vdash C\ \mathtt{type}}
{\Gamma.0\vdash \operatorname{rec}_0 : C}
\end{mathpar}

\item[The $1$-type]

\begin{mathpar}
\inferrule*[right=$1$-formation]
{\Gamma\ \mathtt{context}}
{\Gamma \vdash 1\ \mathtt{type}}

\inferrule*[right=$*$-intro]
{\Gamma\ \mathtt{context}}
{\Gamma \vdash * : 1}

\inferrule*[right=$1$-elimination]
{\Gamma.0 \vdash C\ \mathtt{type} \\ \Gamma \vdash c_0 : C[*\eval]}
{\Gamma.1\vdash \operatorname{rec}_1 \, c_0 : C}
\end{mathpar}

\end{description}

\subsection{Context and substitution constants}
For each natural number $n$, and for each $i$,$j$ and $n$, such that $i < j\le n+2$
\begin{mathpar}
\inferrule*[right=$\Delta_n$-form]
    { }
{\Delta_n \ \mathtt{context}}

\inferrule*[right=$d_i^n$-intro]
    { }
{d_i^n : \Delta_n \to \Delta_{n+1}}

\label{simplicia-identity}
\inferrule*[right=$d_i^n$-comp]
    { }
{d_i^{n+1} \circ d_j^n \equiv d_{j-1}^{n+1} \circ d_i^n: \Delta_n \to \Delta_{n+2}}
\end{mathpar}

\begin{align*}
  \bb{\Delta_n} &= \Delta_n&
  \bb{ d_i^n} &= d_i^n
\end{align*}

\subsection{Soundness}

\begin{theorem}
\label{theorem:soundness}
The intended interpretation of the type theory is sound. That is, all rules for equality holds true in the interpretation given by $\bb{-}$.
\end{theorem}

\begin{proof}
The equalities for substitution are verified in Lemma \ref{Lemma: Soundness for substitution}, and  the rules for $\Pi$ in Lemma \ref{Lemma: Pi regler}. The remaining equations are routine verification.  



\end{proof}

\subsection{Instance specification as type introduction}\label{subsection:Instance specification as type introduction}

The intended interpretation of $\Gamma \vdash A\ \mathtt{type}$ is that $A$ is an instance of the schema $\Gamma$. But context extension allows us to view every instance as a schema in its own right. So for every instance $\Gamma \vdash A\ \mathtt{type}$, we get a schema $\Gamma.A$. It turns out that the most convenient way to specify a schema, is by introducing a new type/instance over one of the simplex schemas $\Delta_n$.

To specify a schema, with a maximum of $n$ attributes, may be seen as introducing a type in the context $\Delta_n$. A relation variable with $k$ attributes in the schema is introduced as an element of the schema substituted into $\Delta_k$. Names of attributes are given as elements of the schema substituted down to $\Delta_0$.

\begin{example}
We construct the rules of the schema $S$ of \ref{Subsubsection: Examples}, with attributes $A$, $B$ and $C$, with two tables $R$ and $Q$. The introduction rules tells us the names of tables and attributes in S. 

\begin{center}
\begin{tabular}{lcr}
$ \Delta_2 \vdash S\ \mathtt{type} $&\hspace{15mm} &$ \Delta_0 \vdash A \equiv R[d_1] : S[d_2 \circ d_1] $\\
$ \Delta_0\vdash A : S[d_2 \circ d_1] $& &$\Delta_0 \vdash B \equiv R[d_0] : S[d_2 \circ d_0]$ \\
$ \Delta_0\vdash B : S[d_2 \circ d_0]$ &  & $\Delta_0 \vdash B \equiv Q[d_1] : S[d_0 \circ d_1] $\\
$ \Delta_0\vdash C : S[d_0 \circ d_0]$ & & $\Delta_0 \vdash C \equiv Q[d_0] : S[d_2 \circ d_0]$ \\
$ \Delta_1 \vdash R : S[d_2] $&  &\\
 $\Delta_1 \vdash Q : S[d_0] $& &
\end{tabular}
\end{center}

From these introduction rules, we can generate an elimination rule. The elimination rule tells us how to construct full tuples in an instance over the schema S.

\begin{mathpar}
 \inferrule*[right=$S$-elim]
  {\Delta_2.  S\vdash I\ \mathtt{type}\\ \Delta_0 \vdash a : I[(d_2 \circ d_1).S] [A\eval] \\ \Delta_0 \vdash b : I[(d_2 \circ d_0).S] [B\eval] \\ \Delta_0 \vdash c : I[(d_0 \circ d_0).S] [C\eval] \\ \Delta_1 \vdash r : I[d_2.S] [R \eval] \\ \Delta_1 \vdash q : I[d_0.S] [Q \eval] \\ \Delta_0 \vdash r[d_1] \equiv a \\  \Delta_0 \vdash r[d_0] \equiv b \\ \Delta_0 \vdash q[d_1] \equiv b \\ \Delta_0 \vdash q[d_0] \equiv c }
 {\Delta_2.S \vdash \operatorname{rec}_S \, a \, b \, c \, r  \, q : I }
\end{mathpar}

Another interpretation of the elimination rule is that it formulates that the schema S contains only what is specified by the above introduction rules; it specifies the schema up to isomorphism.

\end{example}

An instance of a schema is a type depending in the context of the schema. Therefore instance specification is completely analoguous to schema specification. The following example shows how to introduce the instance $I$ of \ref{Subsubsection: Examples}.

\begin{example}
Let $S$ be the schema from the previous example. The following set of introductions presents an instance $I$ of $S$.

\begin{center}
\begin{tabular}{lcl}
$ \Delta_2.S \vdash I\ \mathtt{type} $& &\\
$ \Delta_0 \vdash a : I[(d_2 \circ d_1).S][A\eval] $&\hspace{15mm} & $\Delta_0 \vdash r_0[d_1] \equiv a : I[(d_2 \circ d_1).S][A\eval] $\\
$ \Delta_0 \vdash a' : I[(d_2 \circ d_1).S][A\eval]$ & & $\Delta_0  \vdash r_0[d_0] \equiv b : I[(d_2 \circ d_1).S][B \eval]$ \\
$ \Delta_0 \vdash b : I[(d_2 \circ d_0).S][B\eval]$ & &$\Delta_0 \vdash r_1[d_1] \equiv a' : I[(d_2 \circ d_1).S][A\eval] $\\
$ \Delta_0 \vdash d : I[(d_2 \circ d_0).S][B\eval] $& &$\Delta_0  \vdash r_1[d_0] \equiv b : I[(d_2 \circ d_1).S][B \eval]$ \\
 $\Delta_0 \vdash c : I[(d_0 \circ d_0).S][C\eval] $& & $\Delta_0 \vdash q_0[d_1] \equiv b : I[(d_2 \circ d_1).S][A\eval] $\\
$ \Delta_0 \vdash e : I[(d_0 \circ d_0).S][C\eval] $& & $\Delta_0 \vdash q_0[d_0] \equiv c : I[(d_2 \circ d_1).S][A\eval] $\\
$ \Delta_1 \vdash r_0 : I[(d_2).S][R\eval]$ & & $\Delta_0  \vdash q_1[d_1] \equiv d : I[(d_2 \circ d_1).S][B \eval]$ \\
$ \Delta_1 \vdash r_1 : I[(d_2).S][R\eval]$ & & $ \Delta_0  \vdash q_1[d_0] \equiv e : I[(d_2 \circ d_1).S][B \eval]$ \\
$ \Delta_1 \vdash q_0 : I[(d_0).S][Q\eval] $\\
$ \Delta_1 \vdash q_1 : I[(d_0).S][Q\eval]$
\end{tabular}
\end{center}

The above is clearly very verbose, and can be compressed, at the cost of loosing control over the naming of attributes, into the following.

\begin{center}
\begin{tabular}{lcl}
$ \Delta_2.S \vdash I\ \mathtt{type}$ & \hspace{15mm} & \\
 $\Delta_1 \vdash r_0 : I[(d_2).S][R\eval]$  & &$\Delta_0  \vdash r_0[d_0] \equiv r_1[d_0] : I[(d_2 \circ d_1).S][B \eval]$ \\
 $\Delta_1 \vdash r_1 : I[(d_2).S][R\eval]$ & &  $\Delta_0  \vdash q_0[d_1] \equiv r_0[d_0] : I[(d_2 \circ d_1).S][B \eval]$ \\
$ \Delta_1 \vdash q_0 : I[(d_0).S][Q\eval]$ \\
$ \Delta_1 \vdash q_1 : I[(d_0).S][Q\eval]$
\end{tabular}
\end{center}

\end{example}

 A motivation for introducing a formal type theory of databases is for it to provide a query language. The  development and analysis of this query language is future work, but we provide here an example of a query formulated in type theory, illustrating the concept.

\begin{example}{\bf{Queries}}\label{example:queries}

Let the schema $S$ and the instance $I$ be as in the previous examples.  A query is represented as a type, with the terms of the type being the result of the query. We ask for the matching tuples of $R$ and $Q$, i.e\  the tuples in the natural join $R\bowtie Q$. 
This query is represented by the $\Pi$-type of $I$ over $S$ (cf.\ Example \ref{Example: natural join}), which is a type in $\Delta_2$, 
\begin{align*}
 \Delta_2 \vdash \Pi_S I\ \mathtt{type}
\end{align*}
To construct elements of this type, we can apply the constructor $\lambda$ to full tuples of $I$. These may in turn be constructed using the elimination rule of $S$. Thus, the result of the query is:
\begin{align*}
  \Delta_2 \vdash \lambda \operatorname{rec}_S &\, a \,  b \, c \, r_0 \, q_0 : \Pi_S I &
  \Delta_2 \vdash \lambda \operatorname{rec}_S &\, a' \,  b \, c \, r_1 \, q_0 : \Pi_S I
\end{align*}
These terms represent the expected full tuples $\pair{a,b,c}$ and $\pair{a',b,c}$.
\end{example}


\addcontentsline{toc}{section}{References}

\bibliography{references}
\bibliographystyle{kluwer}

\newpage

\section{Appendix}

This appendix has two parts. Part 1 contains several lemmas from Section \ref{section: the model} the proofs of which were omitted in the main text (in the second lemma below we have kept here an earlier version with certain calculations in the statement of the lemma.) Part 2 is a fuller statement of the type theory in Section \ref{Section: The type theory} (and also includes  some heuristic remarks that were omitted from the main text.)

\section*{Part 1}

\begin{lemma}[Lemma \ref{Lemma: relational instances are based posets}]
Let $X$ a (simplicial) schema  and $F:X\to\xalg{FinSet}$ be a functor. Then $\int_X F$ is a (simplicial) schema  and $p:\int_X F\rightarrow X$ a display morphism if and only if $F$ is a relational instance.
\begin{proof} Let $F$ be a relational instance. It is clear that there is at most one morphism between any two objects in $\int_X F$, and it is easy to see that \op{(\int_X F)} is a based poset with base $\cterm{\pair{A,a}\in X\times \set}{A\in X_0,\ a\in F(A)}$, satisfying the condition for being a simplicial complex. Noticing that $\op{(\int_X F)}_n=\cterm{\pair{x,c}\in X\times \set}{x\in X_n,\ c\in F(x)}$ we see that the projection $p:\int_X F\rightarrow X$ is a display morphism.

Conversely, if $F$ is not relational, then condition 1.b of Definition \ref{Def: based poset}  is violated (by any two keys with the same data).
\end{proof}
\end{lemma}

\begin{lemma}[Lemma \ref{Lemma: Soundness for substitution}]
The following equations hold:
\begin{enumerate}
\item For $X$ in \thry{S} and $I\in \Rel{X}$ and $t\in \Trm{X}{I}$ we have $p\circ \hat{t}=\mathrm{id}_X$.
\item For  $X\in \thry{S}$  and $J\in \Rel{X}$ and $t\in \Trm{X}{J}$ we have $t=v[\hat{t}]$.
\item For $f:X\to<150>Y$ in \thry{S} and $J\in \Rel{Y}$ we have $\kortinlpil{p\circ \tilde{f}= f\circ p}{\kanskje{X}{J[f]}}{Y}$.
\item For $f:X\to<150>Y$ in \thry{S} and $J\in \Rel{Y}$ and $t\in \Trm{Y}{J}$ we have $\kortinlpil{\tilde{f}\circ\widehat{t[f]}=\hat{t}\circ f}{X}{\kanskje{Y}{J}}$.
\item For $f:X\to<150>Y$ in \thry{S} and $J\in \Rel{Y}$ we have $v_J[\tilde{f}](x,a)=v_J(fx,a)=v_{J[f]}(x,a)$, whence $v_J[\tilde{f}]=v_{J[f]}$.
\item For $f:X\to<150>Y$ and $g:Y\to<150>Z$  in \thry{S} and $J\in \Rel{Z}$ we have $\tilde{g}(\tilde{f}(x,a))=\tilde{g}(fx,a)=(gfx,a)=\widetilde{g\circ f}(x,a)$ whence $\widetilde{g\circ f}=\tilde{g}\circ \tilde{f}$.
\item For $X \in \thry{S}$ and $J\in \Rel{X}$ we have $\tilde{p}\circ \hat{v}(x,a)=\tilde{p}(\pair{x,a},a)=\pair{x,a}$ whence $\tilde{p}\circ \hat{v}=\mathrm{Id}_{\kanskje{X}{J}}$. 
\end{enumerate}
\begin{proof}
The first two are immediate. For the third, $v[\hat{t}]$ is the full tuple of $J[p][\hat{t}]=J[p\circ \hat{t}]=J[\mathrm{id}_X]$ defined by $x\mapsto v(\hat{t}(x))=v(\pair{x,t(x)})=t(x)$. For the fourth, we have $p\circ \tilde{f}(x,a)=p(f(x),a)=f(x)=f\circ p(x,a)$. Finally for the fifth, we have $\tilde{f}\circ\widehat{t[f]}(x)=\tilde{f}(x,t(fx))=\pair{f(x),t(fx)}=\hat{t}(fx)=\hat{t}\circ f(x)$.
\end{proof}
\end{lemma}

\begin{lemma}[Lemma \ref{Lemma: display is tuple preserving}]
Let  $f:X\to<150>Y$ in \thry{S}. Then $f$ is display if and only if for all $J\in \Rel{Y}$ on tuple form $J[f]$ is also on tuple form.
\begin{proof} Suppose $f$ is display. For $x\in X$, we have $J[f](x)=J[fx]$. Since $f$ is display and thus restricts to a bijection $B_{\leq x}\to<150> B_{\leq fx}$, $J[f](x)=J[fx]$ is a set of tuples on $\Pi_{A\in B_{\leq x}}J\circ f(A)$. Conversely, suppose $f$ is not display. Then there exists $x\in X$ such that $x$ has two attributes, $B_{\leq x}=\{A,B\}$, and  $f(A)=f(B)=f(x)$. Let $J\in \Rel{Y}$ be an instance such that $J(fA)=\{a\}$. Then $J\circ f(x)=\{a\}$ and not the tuple \pair{a,a}.      
 
\end{proof}
\end{lemma}

\begin{lemma}[Lemma \ref{Lemma: Pi regler}]
 Let $\kortinlpil{f}{Z}{X}$ in be a display morphism in  \thry{S}, $J\in \Rel{X}$,  $G\in\Rel{\kanskje{X}{J}}$, and $t\in \Trm{\kanskje{X}{J}}{G}$.
 \begin{enumerate}
 \item $\mathrm{Ap}_{\lambda t}= t$.
 \item $(\Pi_JG)[f]=\Pi_{J[f]}G[\tilde{f}]$.
 \item $(\lambda t)[f]=\lambda(t[\tilde{f}])$.
 \end{enumerate}
\begin{proof}
\begin{enumerate}
\item (With reference to $\Xi_JG$) we have that $\mathrm{Ap}_{\lambda t}(x,a)= \lambda t(x)_{x,a}=t(x,a) $, whence $\mathrm{Ap}_{\lambda t}= t$.
\item Let $A$ be an attribute of $Z$. Then 
\begin{align*}
(\Pi_JG)[f](A)&=\Pi_JG(fA)=\Pi_{a\in J(fA)}G(fA,a)=\Pi_{a\in J[f](A)}G[\tilde{f}](A,a)\\
                                           &=\Pi_{J[f]}G[\tilde{f}](A)
\end{align*}
Hence it suffices to show that $(\Xi_JG)[f]\cong\Xi_{J[f]}G[\tilde{f}]$.   Let $c\in \Pi_{J[f]}G[\tilde{f}]$. Thus $c$ is a function assigning for each $z'\leq z$ and $a\in J[f](z')=J(fz')$ an element $c_{z',a}\in G[\tilde{f}](z',a)=G(fz',a)$ such that condition (1) is satisfied. On the other hand, an element $c\in (\Xi_JG)[f](z)$ is an element of $\Pi_JG(fz)$, and thus an a function which for every $y\leq fz'$ and $a\in J(y)$ assigns an element $c_{y,a}\in G(y,a)$ satisfying condition (1). The isomorphism now follows from $f$ being display, since then the function 
\begin{equation}\label{Eq: 2}
\bfig  \morphism|a|<2500,0>[P_z=\cterm{\pair{z',a}}{z'\leq z,\ a\in J(A(fz')) }`\cterm{\pair{y,a}}{y\leq fz,\ a\in J(y)}=P_{fz};\pair{z',a}\mapsto\pair{fz',a}]  \efig\end{equation}   
is bijective.
\item We work with the definition of $\Xi$ and show that the two full tuples commute with the isomorphism induced by (\ref{Eq: 2}). But this follows by inspection since,  for $z\in Z$, we have $(\lambda t)[f](z)= \lambda t(fz)=(t(y,a))_{P_{fz}}$ and    $(\lambda(t[\tilde{f}])(z)=(t[\tilde{f}](z',a))_{P_{z}}=(t(fz',a))_{P_z}$.
\end{enumerate}
\end{proof}
\end{lemma}

\section*{Part 2}

\renewcommand{\showelim}{true}
\subsection{Substitutions}

The following axioms state that contexts and substitutions form a category, acting on the types and elements. Below the axioms are stated the intended interpretations cf.\ Section \ref{section: the model}


\begin{mathpar} 
\inferrule*[right=subst-comp]
  {\Gamma,\Delta,\Theta\ \mathtt{context} \\ \sigma : \Gamma \to \Delta \\ \tau : \Delta \to \Theta}
  {\tau \circ \sigma : \Gamma \to \Theta}

\elim{
\inferrule*[right=subst-assoc]
  {\Gamma,\Delta, \Theta, \Lambda\ \mathtt{context} \\ \sigma : \Gamma \to \Delta \\ \tau : \Delta \to \Theta \\ \upsilon : \Theta \to \Lambda}
  {(\upsilon \circ \tau) \circ \sigma \equiv \upsilon \circ (\tau \circ \sigma): \Gamma \to \Lambda}
}


\inferrule*[right=subst-id]
    { \Gamma\ \mathtt{context}}
{id_\Gamma : \Gamma \to \Gamma }


\elim{
\inferrule*[right=subst-id-left]
    { \Gamma,\Delta\ \mathtt{context} \\ \sigma : \Gamma \to \Delta}
{id_\Gamma \circ \sigma \equiv \sigma: \Gamma \to \Delta }
}


\elim{
\inferrule*[right=subst-id-right]
    { \Gamma,\Delta\ \mathtt{context} \\ \sigma : \Gamma \to \Delta}
{\sigma \circ id_\Gamma  \equiv \sigma: \Gamma \to \Delta }
}


\inferrule*[right=type-transfer]
    { \elim{\Gamma, \Delta\ \mathtt{context} \\} \Gamma \vdash A \ \mathtt{type} \\ \sigma : \Delta \to \Gamma }
{\Delta \vdash A[\sigma]\ \mathtt{type} }


\elim{
\inferrule*[right=type-transfer-comp]
    { \elim{\Gamma, \Delta, \Theta\ \mathtt{context} \\} \Gamma \vdash A \ \mathtt{type} \\ \sigma : \Delta \to \Gamma \\ \tau : \Theta \to \Delta }
{\Theta \vdash A[\sigma \circ \tau] \equiv A[\sigma][\tau] }
}


\elim{
\inferrule*[right=type-transfer-id]
    { \elim{\Gamma\ \mathtt{context} \\} \Gamma \vdash A \ \mathtt{type} }
{\Gamma \vdash A[id_\Gamma] \equiv A }
}


\inferrule*[right=term-transfer]
    {\elim{ \Gamma, \Delta\ \mathtt{context} \\ \Gamma \vdash A \ \mathtt{type} \\} \Gamma \vdash a : A \\ \sigma : \Delta \to \Gamma}
{\Delta \vdash a[\sigma] : A[\sigma] }


\elim{
\inferrule*[right=term-transfer-comp]
    { \elim{\Gamma, \Delta, \Theta\ \mathtt{context} \\ \Gamma \vdash A \ \mathtt{type} \\} \Gamma \vdash a : A \\  \sigma : \Delta \to \Gamma \\ \tau : \Theta \to \Delta }
{\Theta \vdash a[\sigma \circ \tau] \equiv a[\sigma][\tau] : A[\sigma][\tau] }
}

\elim{
\inferrule*[right=term-transfer-id]
    { \elim{\Gamma\ \mathtt{context} \\ \Gamma \vdash A \ \mathtt{type} \\} \Gamma \vdash a : A }
{\Delta \vdash a[id_\Gamma] \equiv a : A }
}

\end{mathpar}


\begin{align*}
\bb{\sigma \circ \tau} &= \bb{\sigma} \circ \bb{\tau} &
\bb{A[\sigma]} &= \bb{A} \circ \bb{\sigma} &
\bb{Id_\Gamma} &= Id_{\bb{\Gamma}}  &
\bb{a[\sigma]} &= \bb{a} \bb{\sigma}
\end{align*}

\subsection{Context extension}

The following rules detail how types in a context extends it.


\begin{mathpar}
\inferrule*[right=cont-ext]
    { \Gamma\ \mathtt{context} \\ \Gamma \vdash A \ \mathtt{type}}
{\Gamma. A \ \mathtt{context}}


\inferrule*[right=projection]
    { \elim{ \Gamma\ \mathtt{context} \\} \Gamma \vdash A \ \mathtt{type}}
{{\downarrow} : \Gamma. A \to \Gamma}
\end{mathpar}
\begin{mathpar}

\inferrule*[right=variable]
    { \elim{ \Gamma\ \mathtt{context} \\} \Gamma \vdash A \ \mathtt{type}}
{\Gamma. A \vdash v : A[\downarrow]}


\inferrule*[right=evaluation]
    { \elim{\Gamma\ \mathtt{context} \\ \Gamma \vdash A \ \mathtt{type} \\} \Gamma \vdash a : A}
{{a\eval} : \Gamma \to \Gamma. A}


\inferrule*[right=section]
    { \elim{\Gamma\ \mathtt{context} \\ \Gamma \vdash A \ \mathtt{type} \\ }\Gamma \vdash a : A}
{\downarrow \circ (a\eval) \equiv id_\Gamma: \Gamma \to \Gamma.}


\inferrule*[right=var-eval]
    { \elim{\Gamma\ \mathtt{context} \\ \Gamma \vdash A \ \mathtt{type} \\} \Gamma \vdash a : A}
{\Gamma \vdash v[a\eval] \equiv a : A}


\inferrule*[right=lifting]
    { \elim {\Gamma,\Delta\ \mathtt{context} \\ }\Gamma \vdash A \ \mathtt{type} \\ \sigma : \Delta \to \Gamma }
{{\sigma.A} : \Delta.A[\sigma] \to \Gamma. A}


\inferrule*[right=project-lift]
    { \elim{\Gamma,\Delta\ \mathtt{context} \\} \Gamma \vdash A \ \mathtt{type} \\ \sigma : \Delta \to \Gamma }
{{\downarrow} \circ (\sigma.A) \equiv \sigma \circ {\downarrow} : \Delta.A[\sigma] \to \Gamma}


\inferrule*[right=variable-lift]
    { \elim{\Gamma,\Delta\ \mathtt{context} \\ }\Gamma \vdash A \ \mathtt{type} \\ \sigma : \Delta \to \Gamma }
{ v[f.A] \equiv v : \Delta.A[\sigma]}


\inferrule*[right=eval-lift]
    { \elim{\Gamma,\Delta\ \mathtt{context} \\ \Gamma \vdash A \ \mathtt{type} \\} \Gamma \vdash a : A \\ \sigma : \Delta \to \Gamma }
{ (\sigma.A) \circ (a[\sigma]\eval) \equiv a\eval \circ \sigma : \Delta \to \Gamma.A}


\inferrule*[right= compose-lift]
    { \elim{\Gamma,\Delta,\Theta\ \mathtt{context} \\} \Gamma \vdash A \ \mathtt{type} \\ \sigma : \Delta \to \Gamma \\ \tau : \Theta \to \Delta}
{(\sigma.A) \circ (\tau.A[\sigma]) \equiv (\sigma \circ \tau) . A : \Theta.A[\sigma \circ \tau] \to \Gamma.A}


\inferrule*[right=diagonal-lift]
    { \elim{\Gamma,\Delta\ \mathtt{context} \\} \Gamma \vdash A \ \mathtt{type} }
{ \downarrow.A \circ v\eval \equiv Id_{\Gamma.A}: \Gamma.A \to \Gamma.A}

\end{mathpar}


\begin{align*}
\bb{\Gamma.A} &= \int_{{\bb{\Gamma}}}\bb{A} &
\bb{\downarrow} &= p &
 \bb{v} &= v &
\bb{a\eval} &= \widehat{\bb a} &
\bb{\sigma.A} &= \widetilde{\bb \sigma}
\end{align*}

\elim{
\subsubsection{Reading explicit substitutitions} 

The visually inclined reader will find it helpful to draw diagrams, to understand the explicit substitutions. The more operationally inclined reader might find the following heuristics and examples useful.

The basic substitutions, $\downarrow$ and $t \eval$ and $\sigma.A$, can be seen as modifying the current context. The projection $\downarrow$ removes an assumption. The evaluation $t \eval$, introduces an assumption. The lifting $\sigma.A$ first removes an assumtion of $A[\sigma]$, then performs the action specified by $\sigma$ and then introduces the assumption $A$. The reading order of substitutions is right to left.

\begin{example}
Assume $\Gamma\ \mathtt{context}$, and that $\Gamma \vdash A\ \mathtt{type}$, and $\Gamma \vdash C\ \mathtt{type}$. Then cosider the judgement $\Gamma.C \vdash A[\downarrow]\ \mathtt{type}$, which is a derivable instance of the type transfer rule. Following the above heuristics, the $[\downarrow]$, removes the $.C$ from the context, allowing us to access $A$ from $\Gamma$.

Let us now extend the example with $\Gamma.A \vdash B\ \mathtt{type}$, and $\Gamma.C \vdash D\ \mathtt{type}$, and $\Gamma.C \vdash a:A[\downarrow]$. Consider now the judgement, $\Gamma.C.D \vdash B[\downarrow.A][a\eval][\downarrow]$. Reading from right to left, we obtain the heuristic: $\downarrow$ removes the assumption $D$(leaving $\Gamma.C$). In this context, $a$ is of type $A[\downarrow]$, so $a\eval$ introduces the assumption $A[\downarrow]$. Finally $({\downarrow}.A)$ removes the assumption $A[\downarrow]$, then applies $\downarrow$ which removes the assumption $C$, and adds the assumption $A$, leaving the context $\Gamma.A$, where indeed $B$ is a type.

\end{example}
}

\subsection{The $\Pi$-type}

\begin{mathpar}
\inferrule*[right=$\Pi$-form]
 { \Gamma \ \mathtt{context} \\ \Gamma \vdash A\ \mathtt{type} \\ \Gamma.A\vdash B\ \mathtt{type}}
{\Gamma \vdash \Pi_A B\ \mathtt{type} }

\inferrule*[right=$\Pi$-subst]
 { \Gamma,\Delta \ \mathtt{context} \\ \Gamma \vdash A\ \mathtt{type} \\ \Gamma.A\vdash B\ \mathtt{type} \\ \sigma : \Delta \to \Gamma}
{\Gamma \vdash \Pi_A B [\sigma] \equiv \Pi_{A[\Sigma]}B[\sigma.A] }

\inferrule*[right=$\Pi$-intro]
 { \Gamma.A \vdash b : B}
{\Gamma \vdash \lambda b : \Pi_A B }

\inferrule*[right=$\Pi$-elim]
 { \Gamma \vdash f : \Pi_A B}
{\Gamma .A \vdash \operatorname{apply} f : B }

\inferrule*[right=$\Pi$-comp]
 { \Gamma.A \vdash b : B}
{\Gamma .A \vdash \operatorname{apply}(\lambda b) \equiv b : B }

\end{mathpar}

\begin{align*}
  \bb{\Pi_A B} &= \Pi_{\bb{A}}\bb{B} &
  \bb{\lambda f} &= \lambda \bb{f} &
  \bb{\operatorname{apply}} &= Ap
\end{align*}

\subsection{Other types}

These types correspond to the constructions with the same names in subsection \ref{sigma-modell}.

\begin{description}
\item[The $\Sigma$-type]

\begin{mathpar}
\inferrule*[right=$\Sigma$-form]
 { \Gamma \ \mathtt{context} \\ \Gamma \vdash A\ \mathtt{type} \\ \Gamma.A\vdash B\ \mathtt{type}}
{\Gamma \vdash \Sigma_A B\ \mathtt{type} }

\inferrule*[right=$\Sigma$-intro]
{ \Gamma.A\vdash B\ \mathtt{type}}
{\Gamma.A.B \vdash \operatorname{pair} : \Sigma_A B [\downarrow][\downarrow]}

\elim{
\inferrule*[right=$\Sigma$-elim]
 { \Gamma.\Sigma_A B \vdash C\ \mathtt{type} \\ \Gamma.A.B \vdash c_0 : C[(\downarrow \circ \downarrow).\Sigma_A B][\operatorname{pair}\eval]}
{\Gamma. \Sigma_A B \vdash \operatorname{rec}_\Sigma \, c_0 : C}

\inferrule*[right=$\Sigma$-comp]
 { \Gamma.\Sigma_A B \vdash C\ \mathtt{type} \\ \Gamma.A.B \vdash c_0 : C[(\downarrow \circ \downarrow).\Sigma_A B][\operatorname{pair}\eval]}
{\Gamma.A.B \vdash (\operatorname{rec}_\Sigma \, c_0)[(\downarrow \circ \downarrow).\Sigma_A B][\operatorname{pair}\eval]  \equiv  c : C[(\downarrow \circ \downarrow).\Sigma_A B][\operatorname{pair}\eval ] }
}
\end{mathpar}

\elim{
\begin{align*}
  \bb{\Sigma_A B} &= \Sigma_{\bb{A}}\bb{B} &
  \bb{\operatorname{pair}} &= \operatorname{pair} \elim{&
  \bb{\operatorname{rec}_\Sigma} &= \operatorname{rec}_\Sigma}
\end{align*}}

\item[The indentity type]

\begin{mathpar}
\inferrule*[right=$Id$-form]
 { \Gamma \ \mathtt{context} \\ \Gamma \vdash A\ \mathtt{type}}
{\Gamma .A .A[\downarrow]\vdash Id_A\ \mathtt{type} }

\inferrule*[right=$Id$-intro]
{ \Gamma \vdash A\ \mathtt{type}}
{\Gamma.A \vdash \operatorname{refl} : Id_A [v \eval]}

\elim{
\inferrule*[right=$Id$-elim]
 { \Gamma.A.A[\downarrow].Id_A \vdash C\ \mathtt{type} \\ \Gamma.A \vdash c_0 : C[(v \eval).Id_A][\operatorname{refl}\eval]}
{\Gamma.A.A[\downarrow].Id_A \vdash \operatorname{rec}_{Id} \, c_0 : C}

\inferrule*[right=$Id$-comp]
 { \Gamma.A.A[\downarrow].Id_A \vdash C\ \mathtt{type} \\ \Gamma.A \vdash c_0 : C[(v \eval).Id_A][\operatorname{refl}\eval]}
{ \Gamma.A \vdash (\operatorname{rec}_{Id} \, c_0)[(v \eval).Id_A][\operatorname{refl}\eval] \equiv c_0 : C[(v \eval).Id_A][\operatorname{refl}\eval] }
}

\end{mathpar}

\elim{
\begin{align*}
  \bb{Id_A} &= Id_{\bb{A}}\bb{B} &
  \bb{\operatorname{refl}} &= \operatorname{refl} \elim{&
  \bb{\operatorname{rec}_{Id}} &= \operatorname{rec}_{Id}}
\end{align*}
}

\item[The $+$-type]

\begin{mathpar}
\inferrule*[right=$+$-form]
 {\elim{ \Gamma \ \mathtt{context} \\} \Gamma \vdash A\ \mathtt{type} \\ \Gamma \vdash B\ \mathtt{type}}
{\Gamma \vdash A+B\ \mathtt{type} }

\inferrule*[right=$+$-intro-left]
{ \elim{\Gamma \ \mathtt{context} \\ \Gamma \vdash A\ \mathtt{type} \\ \Gamma \vdash B\ \mathtt{type}}}
{\Gamma.A \vdash \operatorname{left} : (A + B)[\downarrow] }

\inferrule*[right=$+$-intro-right]
{ \elim{\Gamma \ \mathtt{context} \\ \Gamma \vdash A\ \mathtt{type} \\ \Gamma \vdash A\ \mathtt{type}}}
{\Gamma.B \vdash \operatorname{right} : (A + B)[\downarrow]}
\elim{
\inferrule*[right=$+$-elim]
 { \Gamma.(A+B) \vdash C\ \mathtt{type} \\ \Gamma.A \vdash c_0 : C[(\downarrow).(A+B)][\operatorname{left}] \\ \Gamma.B \vdash c_1 : C[(\downarrow).(A+B)][\operatorname{right}] }
{ \Gamma.(A+B) \vdash \operatorname{rec}_{+}\, c_0 \, c_1 : C}

\inferrule*[right=$+$-comp-left]
 {  \Gamma.(A+B) \vdash C\ \mathtt{type} \\ \Gamma.A \vdash c_0 : C[(\downarrow).(A+B)][\operatorname{left} \eval] \\ \Gamma.B \vdash c_1 : C[(\downarrow).(A+B)][\operatorname{right} \eval] }
{ \Gamma.A \vdash (\operatorname{rec}_{+}\, c_0 \, c_1)[(\downarrow).(A+B)][\operatorname{left}\eval] \equiv c_0 : C[(\downarrow).(A+B)][\operatorname{left}\eval] }

\inferrule*[right=$+$-comp-right]
 {  \Gamma.(A+B) \vdash C\ \mathtt{type} \\ \Gamma.A \vdash c_0 : C[(\downarrow).(A+B)][\operatorname{left} \eval] \\ \Gamma.B \vdash c_1 : C[(\downarrow).(A+B)][\operatorname{right} \eval] }
{ \Gamma.B \vdash (\operatorname{rec}_{+}\, c_0 \, c_1)[(\downarrow).(A+B)][\operatorname{right}\eval] \equiv c_1 : C[(\downarrow).(A+B)][\operatorname{right}\eval] }
}
\end{mathpar}

\item[The $0$-type]

\begin{mathpar}
\inferrule*[right=$0$-formation]
{\Gamma\ \mathtt{context}}
{\Gamma \vdash 0\ \mathtt{type}}

\inferrule*[right=$0$-elimination]
{\Gamma.0 \vdash C\ \mathtt{type}}
{\Gamma.0\vdash \operatorname{rec}_0 : C}
\end{mathpar}

\item[The $1$-type]

\begin{mathpar}
\inferrule*[right=$1$-formation]
{\Gamma\ \mathtt{context}}
{\Gamma \vdash 1\ \mathtt{type}}

\inferrule*[right=$*$-intro]
{\Gamma\ \mathtt{context}}
{\Gamma \vdash * : 1}

\inferrule*[right=$1$-elimination]
{\Gamma.0 \vdash C\ \mathtt{type} \\ \Gamma \vdash c_0 : C[*\eval]}
{\Gamma.1\vdash \operatorname{rec}_1 \, c_0 : C}
\end{mathpar}

\end{description}

\subsection{Context and substitution constants}
For each natural number $n$, and for each $i$,$j$ and $n$, such that $i < j\le n+2$
\begin{mathpar}
\inferrule*[right=$\Delta_n$-form]
    { }
{\Delta_n \ \mathtt{context}}

\inferrule*[right=$d_i^n$-intro]
    { }
{d_i^n : \Delta_n \to \Delta_{n+1}}

\label{simplicia-identity}
\inferrule*[right=$d_i^n$-comp]
    { }
{d_i^{n+1} \circ d_j^n \equiv d_{j-1}^{n+1} \circ d_i^n: \Delta_n \to \Delta_{n+2}}
\end{mathpar}

\begin{align*}
  \bb{\Delta_n} &= \Delta_n&
  \bb{ d_i^n} &= d_i^n
\end{align*}

\end{document}